\documentclass{amsart}
\usepackage{amsfonts}

\newtheorem{theorem}{Theorem}
\theoremstyle{plain}

\newtheorem{corollary}{Corollary}

\newtheorem{definition}{Definition}
\newtheorem{example}{Example}

\newtheorem{lemma}{Lemma}

\newtheorem{proposition}{Proposition}

\numberwithin{equation}{section}
\input{tcilatex}

\begin{document}
\title[On Harmonically $s$-Convex Functions]{Ostrowski type inequalities for
harmonically $s$-convex functions}
\author{\.{I}mdat \.{I}\c{s}can}
\address{Department of Mathematics, Faculty of Arts and Sciences,\\
Giresun University, 28100, Giresun, Turkey.}
\email{imdati@yahoo.com}

\begin{abstract}
The author introduces the concept of harmonically $s$-convex functions and
establishes some Ostrowski type inequalities and Hermite-Hadamard type
inequality of these classes of functions.
\end{abstract}

\maketitle

\section{Introduction}

Let $f:I\mathbb{\rightarrow R}$, where $I\subseteq \mathbb{R}$ is an
interval, be a mapping differentiable in $I^{\circ }$ (the interior of $I$)
and let $a,b\in I^{\circ }$ with $a<b.$ If $\left\vert f^{\prime
}(x)\right\vert \leq M,$ for all $x\in \left[ a,b\right] ,$ then the
following inequality holds%
\begin{equation}
\left\vert f(x)-\frac{1}{b-a}\int_{a}^{b}f(t)dt\right\vert \leq M(b-a)\left[ 
\frac{1}{4}+\frac{\left( x-\frac{a+b}{2}\right) ^{2}}{\left( b-a\right) ^{2}}%
\right]  \label{1-1}
\end{equation}%
for all $x\in \left[ a,b\right] .$ This inequality is known in the
literature as the Ostrowski inequality (see \cite{O38}), which gives an
upper bound for the approximation of the integral average $\frac{1}{b-a}%
\int_{a}^{b}f(t)dt$ by the value $f(x)$ at point\ $x\in \left[ a,b\right] $.
For some results which generalize, improve and extend the inequalities(\ref%
{1-1}) we refer the reader to the recent papers (see \cite{ADDC10,L12} ).

In \cite{HM94}, Hudzik and Maligranda considered the following class of
functions:

\begin{definition}
A function $f:I\subseteq 
\mathbb{R}
_{+}\rightarrow 
\mathbb{R}
$ where $%
\mathbb{R}
_{+}=\left[ 0,\infty \right) $, is said to be $s$-convex in the second sense
if%
\begin{equation*}
f\left( \alpha x+\beta y\right) \leq \alpha ^{s}f(x)+\beta ^{s}f(y)
\end{equation*}%
for all $x,y\in I$ and $\alpha ,\beta \geq 0$ with $\alpha +\beta =1$ and $s$
fixed in $\left( 0,1\right] $. They denoted this by $K_{s}^{2}.$
\end{definition}

It can be easily seen that for $s=1$, $s$-convexity reduces to ordinary
convexity of functions defined on $[0,\infty )$.

In \cite{DF99}, Dragomir and Fitzpatrick proved a variant of
Hermite-Hadamard inequality which holds for the $s$-convex functions.

\begin{theorem}
Suppose that $f:%
\mathbb{R}
_{+}\mathbb{\rightarrow }%
\mathbb{R}
_{+}$ is an $s$-convex function in the second sense, where $s\in \lbrack
0,1) $ and let $a,b\in \lbrack 0,\infty )$, $a<b$. If $f\in L\left[ a,b%
\right] $, then the following inequalities hold 
\begin{equation}
2^{s-1}f\left( \frac{a+b}{2}\right) \leq \frac{1}{b-a}\dint%
\limits_{a}^{b}f(x)dx\leq \frac{f(a)+f(b)}{s+1}\text{.}  \label{1-2}
\end{equation}%
the constant $k=\frac{1}{s+1}$ is the best possible in the second inequality
in (\ref{1-2}).
\end{theorem}

The above inequalities are sharp. For recent results and generalizations
concerning $s$-convex functions see \cite{ADK11,DF99,HBI09,I13b,KBO07}.

In \cite{I13c}, the author gave harmonically convex and established
Hermite-Hadamard's inequality for harmonically convex functions as follows:

\begin{definition}
Let $I\subset 
\mathbb{R}
\backslash \left\{ 0\right\} $ be a real interval. A function $%
f:I\rightarrow 
\mathbb{R}
$ is said to be harmonically convex, if \ 
\begin{equation}
f\left( \frac{xy}{tx+(1-t)y}\right) \leq tf(y)+(1-t)f(x)  \label{1-3}
\end{equation}%
for all $x,y\in I$ and $t\in \lbrack 0,1]$. If the inequality in (\ref{1-3})
is reversed, then $f$ is said to be harmonically concave.
\end{definition}

\begin{theorem}
Let $f:I\subset 
\mathbb{R}
\backslash \left\{ 0\right\} \rightarrow 
\mathbb{R}
$ be a harmonically convex function and $a,b\in I$ with $a<b.$ If $f\in
L[a,b]$ then the following inequalities hold 
\begin{equation}
f\left( \frac{2ab}{a+b}\right) \leq \frac{ab}{b-a}\dint\limits_{a}^{b}\frac{%
f(x)}{x^{2}}dx\leq \frac{f(a)+f(b)}{2}.  \label{1-4}
\end{equation}%
The \ above inequalities are sharp.
\end{theorem}

The goal of this paper is to introduce the concept of the harmonically $s$%
-convex functions, obtain the similar the inequalities (\ref{1-4}) for
harmonically $s$-convex functions and establish some new inequalities of
Ostrowski type for harmonically $s$-convex functions.

\section{Main Results}

\begin{definition}
Let $I\subset \left( 0,\infty \right) $ be an real interval. A function $%
f:I\rightarrow 
\mathbb{R}
$ is said to be harmonically $s-$convex (concave), if \ 
\begin{equation}
f\left( \frac{xy}{tx+(1-t)y}\right) \leq \left( \geq \right)
t^{s}f(y)+(1-t)^{s}f(x)  \label{2-1}
\end{equation}%
for all $x,y\in I$ , $t\in \lbrack 0,1]$ and for some fixed $s\in \left( 0,1%
\right] $.
\end{definition}

\begin{proposition}
Let $I\subset \left( 0,\infty \right) $ be an real interval and $%
f:I\rightarrow 
\mathbb{R}
$ is a function, then ;

\begin{enumerate}
\item if f is $s$-convex and nondecreasing function then f is harmonically $s
$-convex.

\item if f is harmonically $s$-convex and nonincreasing function then f is $s
$-convex.
\end{enumerate}
\end{proposition}

\begin{proof}
Since $f:\left( 0,\infty \right) \rightarrow 
\mathbb{R}
,\ f(x)=x,$ harmonically convex function, we have 
\begin{equation}
\frac{xy}{tx+(1-t)y}\leq ty+(1-t)x  \label{2-1a}
\end{equation}%
for all $x,y\in \left( 0,\infty \right) $ , $t\in \lbrack 0,1]$ (see also 
\cite[page 4]{CD11}). The proposition (1) and (2) is easily obtained from
the inequality (\ref{2-1a}).
\end{proof}

\begin{example}
Let $s\in \left( 0,1\right] $ and $f:\left( 0,1\right] \rightarrow \left( 0,1%
\right] ,\ f(x)=x^{s}$. Since $f$ is $s$-convex (see \ \cite{HM94}) and
nondecreasing function, f is harmonically $s-$convex.
\end{example}

The following result of the Hermite-Hadamard type holds.

\begin{theorem}
\label{2.2}Let $f:I\subset \left( 0,\infty \right) \rightarrow 
\mathbb{R}
$ be an harmonically $s$-convex function and $a,b\in I$ with $a<b.$ If $f\in
L[a,b]$ then the following inequalities hold:%
\begin{equation}
2^{s-1}f\left( \frac{2ab}{a+b}\right) \leq \frac{ab}{b-a}\dint\limits_{a}^{b}%
\frac{f(x)}{x^{2}}dx\leq \frac{f(a)+f(b)}{s+1}.  \label{2-2}
\end{equation}
\end{theorem}

\begin{proof}
Since $f:I\rightarrow 
\mathbb{R}
$ is an harmonically $s$-convex function, we have, for all $x,y\in I$ (with $%
t=\frac{1}{2}$ in the inequality (\ref{2-1}) )%
\begin{equation*}
f\left( \frac{2xy}{x+y}\right) \leq \frac{f(y)+f(x)}{2^{s}}
\end{equation*}%
Choosing $x=\frac{ab}{ta+(1-t)b},\ y=\frac{ab}{tb+(1-t)a}$, we get%
\begin{equation*}
f\left( \frac{2ab}{a+b}\right) \leq \frac{f\left( \frac{ab}{tb+(1-t)a}%
\right) +f\left( \frac{ab}{ta+(1-t)b}\right) }{2^{s}}
\end{equation*}%
Further, integrating for $t\in \lbrack 0,1]$, we have%
\begin{equation}
f\left( \frac{2ab}{a+b}\right) \leq \frac{1}{2^{s}}\left[ \dint%
\limits_{0}^{1}f\left( \frac{ab}{tb+(1-t)a}\right)
dt+\dint\limits_{0}^{1}f\left( \frac{ab}{ta+(1-t)b}\right) dt\right] 
\label{2-2a}
\end{equation}%
Since each of the integrals is equal to $\frac{ab}{b-a}\dint\limits_{a}^{b}%
\frac{f(x)}{x^{2}}dx$, we obtain the left-hand side of the inequality (\ref%
{2-2}) from (\ref{2-2a}).

The proof of the second inequality follows by using (\ref{2-1}) with $x=a$
and $y=b$ and integrating with respect to $t$ over $[0,1]$.
\end{proof}

In order to prove our main theorems, we need the following lemma:

\begin{lemma}
\label{2.1}Let $f:I\subset 
\mathbb{R}
\backslash \left\{ 0\right\} \rightarrow 
\mathbb{R}
$ be a differentiable function on $I^{\circ }$ and $a,b\in I$ with $a<b$. If 
$f^{\prime }\in L[a,b]$ then 
\begin{eqnarray*}
&&f(x)-\frac{ab}{b-a}\dint\limits_{a}^{b}\frac{f(u)}{u^{2}}du \\
&=&\frac{ab}{b-a}\left\{ \left( x-a\right) ^{2}\dint\limits_{0}^{1}\frac{t}{%
\left( ta+(1-t)x\right) ^{2}}f^{\prime }\left( \frac{ax}{ta+(1-t)x}\right)
dt\right. \\
&&-\left. \left( b-x\right) ^{2}\dint\limits_{0}^{1}\frac{t}{\left(
tb+(1-t)x\right) ^{2}}f^{\prime }\left( \frac{bx}{tb+(1-t)x}\right)
dt\right\}
\end{eqnarray*}
\end{lemma}

\begin{proof}
Integrating by part and changing variables of integration yields%
\begin{eqnarray*}
&&\frac{ab}{b-a}\left\{ \left( x-a\right) ^{2}\dint\limits_{0}^{1}\frac{t}{%
\left( ta+(1-t)x\right) ^{2}}f^{\prime }\left( \frac{ax}{ta+(1-t)x}\right)
dt\right.  \\
&&-\left. \left( b-x\right) ^{2}\dint\limits_{0}^{1}\frac{t}{\left(
tb+(1-t)x\right) ^{2}}f^{\prime }\left( \frac{bx}{tb+(1-t)x}\right)
dt\right\} 
\end{eqnarray*}%
\begin{eqnarray*}
&=&\frac{1}{x(b-a)}\left[ b\left( x-a\right) \dint\limits_{0}^{1}tdf\left( 
\frac{ax}{ta+(1-t)x}\right) +a\left( b-x\right)
\dint\limits_{0}^{1}tdf\left( \frac{bx}{tb+(1-t)x}\right) \right]  \\
&=&\frac{1}{x(b-a)}\left[ b\left( x-a\right) \left\{ \left. tf\left( \frac{ax%
}{ta+(1-t)x}\right) \right\vert _{0}^{1}-\dint\limits_{0}^{1}f\left( \frac{ax%
}{ta+(1-t)x}\right) dt\right\} \right]  \\
&&+\frac{1}{x(b-a)}\left[ a\left( b-x\right) \left\{ \left. tf\left( \frac{bx%
}{tb+(1-t)x}\right) \right\vert _{0}^{1}-\dint\limits_{0}^{1}f\left( \frac{bx%
}{tb+(1-t)x}\right) dt\right\} \right]  \\
&=&f(x)-\frac{ab}{b-a}\dint\limits_{a}^{b}\frac{f(u)}{u^{2}}du.
\end{eqnarray*}
\end{proof}

\begin{theorem}
\label{2.3}Let $f:I\subset \left( 0,\infty \right) \rightarrow 
\mathbb{R}
$ be a differentiable function on $I^{\circ }$, $a,b\in I$ with $a<b,$ and $%
f^{\prime }\in L[a,b].$ If $\left\vert f^{\prime }\right\vert ^{q}$ is
harmonically $s$- convex on $[a,b]$ for $q\geq 1,$ then for all $x\in \left[
a,b\right] $, we have%
\begin{equation}
\left\vert f(x)-\frac{ab}{b-a}\dint\limits_{a}^{b}\frac{f(u)}{u^{2}}%
du\right\vert   \label{2-3}
\end{equation}%
\begin{eqnarray*}
&\leq &\frac{ab}{b-a}\left\{ \left( x-a\right) ^{2}\left( \lambda
_{1}(a,x,s,q,q)\left\vert f^{\prime }\left( x\right) \right\vert
^{q}+\lambda _{2}((a,x,s,q,q)\left\vert f^{\prime }\left( a\right)
\right\vert ^{q}\right) ^{\frac{1}{q}}\right.  \\
&&+\left. \left( b-x\right) ^{2}\left( \lambda _{3}(b,x,s,q,q)\left\vert
f^{\prime }\left( x\right) \right\vert ^{q}+\lambda
_{4}(b,x,s,q,q)\left\vert f^{\prime }\left( b\right) \right\vert ^{q}\right)
^{\frac{1}{q}}\right\} ,
\end{eqnarray*}%
where 
\begin{equation*}
\lambda _{1}(a,x,s,\vartheta ,\rho )=\frac{\beta \left( \rho +s+1,1\right) }{%
x^{2\vartheta }}._{2}F_{1}\left( 2\vartheta ,\rho +s+1;\rho +s+2;1-\frac{a}{x%
}\right) ,
\end{equation*}%
\begin{equation*}
\lambda _{2}(a,x,s,\vartheta ,\rho )=\frac{\beta \left( \rho +1,1\right) }{%
x^{2q}}._{2}F_{1}\left( 2q,\rho +1;\rho +s+2;1-\frac{a}{x}\right) ,
\end{equation*}%
\begin{equation*}
\lambda _{3}(b,x,s,\vartheta ,\rho )=\frac{\beta \left( 1,\rho +s+1\right) }{%
b^{2\vartheta }}._{2}F_{1}\left( 2\vartheta ,1;\rho +s+2;1-\frac{x}{b}%
\right) ,
\end{equation*}%
\begin{equation*}
\lambda _{4}(b,x,s,\vartheta ,\rho )=\frac{\beta \left( s+1,\rho +1\right) }{%
b^{2\vartheta }}._{2}F_{1}\left( 2\vartheta ,s+1;\rho +s+2;1-\frac{x}{b}%
\right) ,
\end{equation*}%
$\beta $ is Euler Beta function defined by%
\begin{equation*}
\beta \left( x,y\right) =\frac{\Gamma (x)\Gamma (y)}{\Gamma (x+y)}%
=\dint\limits_{0}^{1}t^{x-1}\left( 1-t\right) ^{y-1}dt,\ \ x,y>0,
\end{equation*}%
and $_{2}F_{1}$ is hypergeometric function defined by 
\begin{equation*}
_{2}F_{1}\left( a,b;c;z\right) =\frac{1}{\beta \left( b,c-b\right) }%
\dint\limits_{0}^{1}t^{b-1}\left( 1-t\right) ^{c-b-1}\left( 1-zt\right)
^{-a}dt,\ c>b>0,\ \left\vert z\right\vert <1\text{ (see \cite{AS65}).}
\end{equation*}
\end{theorem}

\begin{proof}
From Lemma \ref{2.1}, Power mean inequality and the harmonically $s$%
-convexity of $\left\vert f^{\prime }\right\vert ^{q}$ on $[a,b],$we have%
\begin{eqnarray*}
&&\left\vert f(x)-\frac{ab}{b-a}\dint\limits_{a}^{b}\frac{f(u)}{u^{2}}%
du\right\vert  \\
&\leq &\frac{ab}{b-a}\left\{ \left( x-a\right) ^{2}\dint\limits_{0}^{1}\frac{%
t}{\left( ta+(1-t)x\right) ^{2}}\left\vert f^{\prime }\left( \frac{ax}{%
ta+(1-t)x}\right) \right\vert dt\right.  \\
&&+\left. \left( b-x\right) ^{2}\dint\limits_{0}^{1}\frac{t}{\left(
tb+(1-t)x\right) ^{2}}\left\vert f^{\prime }\left( \frac{bx}{tb+(1-t)x}%
\right) \right\vert dt\right\} 
\end{eqnarray*}%
\begin{eqnarray}
&\leq &\frac{ab\left( x-a\right) ^{2}}{b-a}\left(
\dint\limits_{0}^{1}1dt\right) ^{1-\frac{1}{q}}  \label{2-3a} \\
&&\times \left( \dint\limits_{0}^{1}\frac{t^{q}}{\left( ta+(1-t)x\right)
^{2q}}\left[ t^{s}\left\vert f^{\prime }\left( x\right) \right\vert
^{q}+(1-t)^{s}\left\vert f^{\prime }\left( a\right) \right\vert ^{q}\right]
dt\right) ^{\frac{1}{q}}  \notag
\end{eqnarray}%
\begin{eqnarray*}
&&+\frac{ab\left( b-x\right) ^{2}}{b-a}\left( \dint\limits_{0}^{1}1dt\right)
^{1-\frac{1}{q}} \\
&&\times \left( \dint\limits_{0}^{1}\frac{t^{q}}{\left( tb+(1-t)x\right)
^{2q}}\left[ t^{s}\left\vert f^{\prime }\left( x\right) \right\vert
^{q}+(1-t)^{s}\left\vert f^{\prime }\left( b\right) \right\vert ^{q}\right]
dt\right) ^{\frac{1}{q}},
\end{eqnarray*}%
where an easy calculation gives%
\begin{equation}
\dint\limits_{0}^{1}\frac{t^{q+s}}{\left( ta+(1-t)x\right) ^{2q}}dt=\frac{%
\beta \left( q+s+1,1\right) }{x^{2q}}._{2}F_{1}\left( 2q,q+s+1;q+s+2;1-\frac{%
a}{x}\right) ,  \label{2-3b}
\end{equation}%
\begin{equation*}
\dint\limits_{0}^{1}\frac{t^{q+s}}{\left( tb+(1-t)x\right) ^{2q}}dt=\frac{%
\beta \left( 1,q+s+1\right) }{b^{2q}}._{2}F_{1}\left( 2q,1;q+s+2;1-\frac{x}{b%
}\right) ,
\end{equation*}%
\begin{equation*}
\dint\limits_{0}^{1}\frac{t^{q}(1-t)^{s}}{\left( ta+(1-t)x\right) ^{2q}}dt=%
\frac{\beta \left( q+1,s+1\right) }{x^{2q}}._{2}F_{1}\left( 2q,q+1;s+q+2;1-%
\frac{a}{x}\right) ,
\end{equation*}%
\begin{equation}
\dint\limits_{0}^{1}\frac{t^{q}(1-t)^{s}}{\left( tb+(1-t)x\right) ^{2q}}dt=%
\frac{\beta \left( s+1,q+1\right) }{b^{2q}}._{2}F_{1}\left( 2q,s+1;s+q+2;1-%
\frac{x}{b}\right) .  \label{2-3c}
\end{equation}%
Hence, If we use (\ref{2-3b})-(\ref{2-3c}) in (\ref{2-3a}), we obtain the
desired result. This completes the proof.
\end{proof}

\begin{corollary}
In Theorem \ref{2.3}, additionally, if $|f^{\prime }(x)|\leq M$, $x\in \left[
a,b\right] ,$ then inequality 
\begin{eqnarray*}
&&\left\vert f(x)-\frac{ab}{b-a}\dint\limits_{a}^{b}\frac{f(u)}{u^{2}}%
du\right\vert  \\
&\leq &\frac{ab}{b-a}M\left\{ \left( x-a\right) ^{2}\left( \lambda
_{1}(a,x,s,q,q)+\lambda _{2}((a,x,s,q,q)\right) ^{\frac{1}{q}}\right.  \\
&&+\left. \left( b-x\right) ^{2}\left( \lambda _{3}(b,x,s,q,q)+\lambda
_{4}(b,x,s,q,q)\right) ^{\frac{1}{q}}\right\} 
\end{eqnarray*}%
holds.
\end{corollary}

\begin{theorem}
\label{2.4}Let $f:I\subset \left( 0,\infty \right) \rightarrow 
\mathbb{R}
$ be a differentiable function on $I^{\circ }$, $a,b\in I$ with $a<b,$ and $%
f^{\prime }\in L[a,b].$ If $\left\vert f^{\prime }\right\vert ^{q}$ is
harmonically $s$- convex on $[a,b]$ for $q\geq 1,$ then for all $x\in \left[
a,b\right] $, we have%
\begin{equation}
\left\vert f(x)-\frac{ab}{b-a}\dint\limits_{a}^{b}\frac{f(u)}{u^{2}}%
du\right\vert   \label{2-4}
\end{equation}%
\begin{eqnarray*}
&\leq &\frac{ab}{b-a}\left( \frac{1}{2}\right) ^{1-\frac{1}{q}}\left\{
\left( x-a\right) ^{2}\left( \lambda _{1}(a,x,s,q,1)\left\vert f^{\prime
}\left( x\right) \right\vert ^{q}+\lambda _{2}(a,x,s,q,1)\left\vert
f^{\prime }\left( a\right) \right\vert ^{q}\right) ^{\frac{1}{q}}\right.  \\
&&+\left. \left( b-x\right) ^{2}\left( \lambda _{3}(b,x,s,q,1)\left\vert
f^{\prime }\left( x\right) \right\vert ^{q}+\lambda
_{4}(b,x,s,q,1)\left\vert f^{\prime }\left( b\right) \right\vert ^{q}\right)
^{\frac{1}{q}}\right\} 
\end{eqnarray*}%
where $\lambda _{1}$, $\lambda _{2}$, $\lambda _{3}$ and $\lambda _{4}$ are
defined as in Theorem \ref{2.3}.
\end{theorem}

\begin{proof}
From Lemma \ref{2.1}, Power mean inequality and the harmonically $s$%
-convexity of $\left\vert f^{\prime }\right\vert ^{q}$ on $[a,b],$we have%
\begin{eqnarray}
&&\left\vert f(x)-\frac{ab}{b-a}\dint\limits_{a}^{b}\frac{f(u)}{u^{2}}%
du\right\vert   \label{2-4a} \\
&\leq &\frac{ab\left( x-a\right) ^{2}}{b-a}\left(
\dint\limits_{0}^{1}tdt\right) ^{1-\frac{1}{q}}  \notag \\
&&\times \left( \dint\limits_{0}^{1}\frac{t}{\left( ta+(1-t)x\right) ^{2q}}%
\left[ t^{s}\left\vert f^{\prime }\left( x\right) \right\vert
^{q}+(1-t)^{s}\left\vert f^{\prime }\left( a\right) \right\vert ^{q}\right]
dt\right) ^{\frac{1}{q}}  \notag
\end{eqnarray}%
\begin{eqnarray*}
&&+\frac{ab\left( b-x\right) ^{2}}{b-a}\left( \dint\limits_{0}^{1}tdt\right)
^{1-\frac{1}{q}} \\
&&\times \left( \dint\limits_{0}^{1}\frac{t}{\left( tb+(1-t)x\right) ^{2q}}%
\left[ t^{s}\left\vert f^{\prime }\left( x\right) \right\vert
^{q}+(1-t)^{s}\left\vert f^{\prime }\left( b\right) \right\vert ^{q}\right]
dt\right) ^{\frac{1}{q}}
\end{eqnarray*}%
\begin{eqnarray*}
&\leq &\frac{ab}{b-a}\left( \frac{1}{2}\right) ^{1-\frac{1}{q}}\left\{
\left( x-a\right) ^{2}\left( \lambda _{1}(a,x,s,q,1)\left\vert f^{\prime
}\left( x\right) \right\vert ^{q}+\lambda _{2}(a,x,s,q,1)\left\vert
f^{\prime }\left( a\right) \right\vert ^{q}\right) ^{\frac{1}{q}}\right.  \\
&&+\left. \left( b-x\right) ^{2}\left( \lambda _{3}(b,x,s,q,1)\left\vert
f^{\prime }\left( x\right) \right\vert ^{q}+\lambda
_{4}(b,x,s,q,1)\left\vert f^{\prime }\left( b\right) \right\vert ^{q}\right)
^{\frac{1}{q}}\right\} 
\end{eqnarray*}
This completes the proof.
\end{proof}

\begin{corollary}
In Theorem \ref{2.4}, additionally, if $|f^{\prime }(x)|\leq M$, $x\in \left[
a,b\right] ,$ then inequality 
\begin{eqnarray*}
&&\left\vert f(x)-\frac{ab}{b-a}\dint\limits_{a}^{b}\frac{f(u)}{u^{2}}%
du\right\vert  \\
&\leq &\frac{ab}{b-a}M\left( \frac{1}{2}\right) ^{1-\frac{1}{q}}\left\{
\left( x-a\right) ^{2}\left( \lambda _{1}(a,x,s,q,1)+\lambda
_{2}((a,x,s,q,1)\right) ^{\frac{1}{q}}\right.  \\
&&+\left. \left( b-x\right) ^{2}\left( \lambda _{3}(b,x,s,q,1)+\lambda
_{4}(b,x,s,q,1)\right) ^{\frac{1}{q}}\right\} 
\end{eqnarray*}%
holds.
\end{corollary}

\begin{theorem}
\label{2.5}Let $f:I\subset \left( 0,\infty \right) \rightarrow 
\mathbb{R}
$ be a differentiable function on $I^{\circ }$, $a,b\in I$ with $a<b,$ and $%
f^{\prime }\in L[a,b].$ If $\left\vert f^{\prime }\right\vert ^{q}$ is
harmonically $s$- convex on $[a,b]$ for $q\geq 1,$ then for all $x\in \left[
a,b\right] $, we have%
\begin{equation}
\left\vert f(x)-\frac{ab}{b-a}\dint\limits_{a}^{b}\frac{f(u)}{u^{2}}%
du\right\vert   \label{2-5}
\end{equation}%
\begin{eqnarray*}
&\leq &\frac{ab}{b-a}\left\{ \lambda _{5}^{1-\frac{1}{q}}(a,x)\left(
x-a\right) ^{2}\left( \lambda _{1}(a,x,s,1,1)\left\vert f^{\prime }\left(
x\right) \right\vert ^{q}+\lambda _{2}(a,x,s,1,1)\left\vert f^{\prime
}\left( a\right) \right\vert ^{q}\right) ^{\frac{1}{q}}\right.  \\
&&+\left. \lambda _{5}^{1-\frac{1}{q}}(b,x)\left( b-x\right) ^{2}\left(
\lambda _{3}(b,x,s,1,1)\left\vert f^{\prime }\left( x\right) \right\vert
^{q}+\lambda _{4}(b,x,s,1,1)\left\vert f^{\prime }\left( b\right)
\right\vert ^{q}\right) ^{\frac{1}{q}}\right\} 
\end{eqnarray*}%
where 
\begin{equation*}
\lambda _{5}(\theta ,x)=\frac{1}{x-\theta }\left\{ \frac{1}{\theta }-\frac{%
\ln x-\ln \theta }{x-\theta }\right\} ,
\end{equation*}%
and $\lambda _{1}$, $\lambda _{2}$, $\lambda _{3}$ and $\lambda _{4}$ are
defined as in Theorem \ref{2.3}.
\end{theorem}

\begin{proof}
From Lemma \ref{2.1}, Power mean inequality and the harmonically $s$%
-convexity of $\left\vert f^{\prime }\right\vert ^{q}$ on $[a,b],$we have%
\begin{eqnarray}
&&\left\vert f(x)-\frac{ab}{b-a}\dint\limits_{a}^{b}\frac{f(u)}{u^{2}}%
du\right\vert   \label{2-5a} \\
&\leq &\frac{ab\left( x-a\right) ^{2}}{b-a}\left( \dint\limits_{0}^{1}\frac{t%
}{\left( ta+(1-t)x\right) ^{2}}dt\right) ^{1-\frac{1}{q}}  \notag \\
&&\times \left( \dint\limits_{0}^{1}\frac{t}{\left( ta+(1-t)x\right) ^{2}}%
\left[ t^{s}\left\vert f^{\prime }\left( x\right) \right\vert
^{q}+(1-t)^{s}\left\vert f^{\prime }\left( a\right) \right\vert ^{q}\right]
dt\right) ^{\frac{1}{q}}  \notag
\end{eqnarray}%
\begin{eqnarray*}
&&+\frac{ab\left( b-x\right) ^{2}}{b-a}\left( \dint\limits_{0}^{1}\frac{t}{%
\left( tb+(1-t)x\right) ^{2}}dt\right) ^{1-\frac{1}{q}} \\
&&\times \left( \dint\limits_{0}^{1}\frac{t}{\left( tb+(1-t)x\right) ^{2}}%
\left[ t^{s}\left\vert f^{\prime }\left( x\right) \right\vert
^{q}+(1-t)^{s}\left\vert f^{\prime }\left( b\right) \right\vert ^{q}\right]
dt\right) ^{\frac{1}{q}}.
\end{eqnarray*}%
It is easily check that%
\begin{equation}
\dint\limits_{0}^{1}\frac{t}{\left( ta+(1-t)x\right) ^{2}}dt=\frac{1}{x-a}%
\left\{ \frac{1}{a}-\frac{\ln x-\ln a}{x-a}\right\} ,  \label{2-5b}
\end{equation}%
\begin{equation*}
\dint\limits_{0}^{1}\frac{t}{\left( tb+(1-t)x\right) ^{2}}dt=\frac{1}{b-x}%
\left\{ \frac{\ln b-\ln x}{b-x}-\frac{1}{b}\right\} ,
\end{equation*}%
Hence, If we use (\ref{2-3b})-(\ref{2-3c}) \ for $q=1$ and (\ref{2-5b}) in (%
\ref{2-5a}), we obtain the desired result. This completes the proof.
\end{proof}

\begin{corollary}
In Theorem \ref{2.5}, additionally, if $|f^{\prime }(x)|\leq M$, $x\in \left[
a,b\right] ,$ then inequality 
\begin{eqnarray*}
&&\left\vert f(x)-\frac{ab}{b-a}\dint\limits_{a}^{b}\frac{f(u)}{u^{2}}%
du\right\vert  \\
&\leq &\frac{ab}{b-a}M\left\{ \lambda _{5}^{1-\frac{1}{q}}(a,x)\left(
x-a\right) ^{2}\left( \lambda _{1}(a,x,s,1,1)+\lambda _{2}(a,x,s,1,1)\right)
^{\frac{1}{q}}\right.  \\
&&+\left. \lambda _{5}^{1-\frac{1}{q}}(b,x)\left( b-x\right) ^{2}\left(
\lambda _{3}(b,x,s,1,1)+\lambda _{4}(b,x,s,1,1)\right) ^{\frac{1}{q}%
}\right\} 
\end{eqnarray*}%
holds.
\end{corollary}

\begin{theorem}
\label{2.6}Let $f:I\subset \left( 0,\infty \right) \rightarrow 
\mathbb{R}
$ be a differentiable function on $I^{\circ }$, $a,b\in I$ with $a<b,$ and $%
f^{\prime }\in L[a,b].$ If $\left\vert f^{\prime }\right\vert ^{q}$ is
harmonically $s$-convex on $[a,b]$ for $q>1,\;\frac{1}{p}+\frac{1}{q}=1,$
then%
\begin{equation}
\left\vert f(x)-\frac{ab}{b-a}\dint\limits_{a}^{b}\frac{f(u)}{u^{2}}%
du\right\vert   \label{2-6}
\end{equation}%
\begin{eqnarray*}
&\leq &\frac{ab}{b-a}\left( \frac{1}{p+1}\right) ^{\frac{1}{p}}\left\{
\left( x-a\right) ^{2}\left( \lambda _{1}(a,x,s,q,0)\left\vert f^{\prime
}\left( x\right) \right\vert ^{q}+\lambda _{2}(a,x,s,q,0)\left\vert
f^{\prime }\left( a\right) \right\vert ^{q}\right) ^{\frac{1}{q}}\right.  \\
&&+\left. \left( b-x\right) ^{2}\left( \lambda _{3}(b,x,s,q,0)\left\vert
f^{\prime }\left( x\right) \right\vert ^{q}+\lambda
_{4}(b,x,s,q,0)\left\vert f^{\prime }\left( b\right) \right\vert ^{q}\right)
^{\frac{1}{q}}\right\} .
\end{eqnarray*}%
where $\lambda _{1}$, $\lambda _{2}$, $\lambda _{3}$ and $\lambda _{4}$ are
defined as in Theorem \ref{2.3}.
\end{theorem}

\begin{proof}
From Lemma \ref{2.1}, H\"{o}lder's inequality and the harmonically convexity
of $\left\vert f^{\prime }\right\vert ^{q}$ on $[a,b],$we have%
\begin{eqnarray*}
&&\left\vert f(x)-\frac{ab}{b-a}\dint\limits_{a}^{b}\frac{f(u)}{u^{2}}%
du\right\vert  \\
&\leq &\frac{ab\left( x-a\right) ^{2}}{b-a}\left(
\dint\limits_{0}^{1}t^{p}dt\right) ^{\frac{1}{p}} \\
&&\times \left( \dint\limits_{0}^{1}\frac{1}{\left( ta+(1-t)x\right) ^{2q}}%
\left[ t^{s}\left\vert f^{\prime }\left( x\right) \right\vert
^{q}+(1-t)^{s}\left\vert f^{\prime }\left( a\right) \right\vert ^{q}\right]
dt\right) ^{\frac{1}{q}}
\end{eqnarray*}%
\begin{eqnarray*}
&&+\frac{ab\left( b-x\right) ^{2}}{b-a}\left(
\dint\limits_{0}^{1}t^{p}dt\right) ^{\frac{1}{p}} \\
&&\times \left( \dint\limits_{0}^{1}\frac{1}{\left( tb+(1-t)x\right) ^{2q}}%
\left[ t^{s}\left\vert f^{\prime }\left( x\right) \right\vert
^{q}+(1-t)^{s}\left\vert f^{\prime }\left( b\right) \right\vert ^{q}\right]
dt\right) ^{\frac{1}{q}}
\end{eqnarray*}%
\begin{eqnarray*}
&\leq &\frac{ab}{b-a}\left( \frac{1}{p+1}\right) ^{\frac{1}{p}}\left\{
\left( x-a\right) ^{2}\left( \lambda _{1}(a,x,s,q,0)\left\vert f^{\prime
}\left( x\right) \right\vert ^{q}+\lambda _{2}(a,x,s,q,0)\left\vert
f^{\prime }\left( a\right) \right\vert ^{q}\right) ^{\frac{1}{q}}\right.  \\
&&+\left. \left( b-x\right) ^{2}\left( \lambda _{3}(b,x,s,q,0)\left\vert
f^{\prime }\left( x\right) \right\vert ^{q}+\lambda
_{4}(b,x,s,q,0)\left\vert f^{\prime }\left( b\right) \right\vert ^{q}\right)
^{\frac{1}{q}}\right\} .
\end{eqnarray*}%
This completes the proof.
\end{proof}

\begin{corollary}
In Theorem \ref{2.6}, additionally, if $|f^{\prime }(x)|\leq M$, $x\in \left[
a,b\right] ,$ then inequality 
\begin{eqnarray*}
&&\left\vert f(x)-\frac{ab}{b-a}\dint\limits_{a}^{b}\frac{f(u)}{u^{2}}%
du\right\vert  \\
&\leq &\frac{ab}{b-a}M\left( \frac{1}{p+1}\right) ^{\frac{1}{p}}\left\{
\left( x-a\right) ^{2}\left( \lambda _{1}(a,x,s,q,0)+\lambda
_{2}(a,x,s,q,0)\right) ^{\frac{1}{q}}\right.  \\
&&+\left. \left( b-x\right) ^{2}\left( \lambda _{3}(b,x,s,q,0)+\lambda
_{4}(b,x,s,q,0)\right) ^{\frac{1}{q}}\right\} 
\end{eqnarray*}%
holds.
\end{corollary}

\begin{theorem}
\label{2.7}Let $f:I\subset \left( 0,\infty \right) \rightarrow 
\mathbb{R}
$ be a differentiable function on $I^{\circ }$, $a,b\in I$ with $a<b,$ and $%
f^{\prime }\in L[a,b].$ If $\left\vert f^{\prime }\right\vert ^{q}$ is
harmonically $s$-convex on $[a,b]$ for $q>1,\;\frac{1}{p}+\frac{1}{q}=1,$
then%
\begin{equation*}
\left\vert f(x)-\frac{ab}{b-a}\dint\limits_{a}^{b}\frac{f(u)}{u^{2}}%
du\right\vert 
\end{equation*}%
\begin{eqnarray*}
&\leq &\frac{ab}{b-a}\left\{ \left( \lambda _{1}(a,x,0,p,p)\right) ^{\frac{1%
}{p}}\left( x-a\right) ^{2}\left( \frac{\left\vert f^{\prime }\left(
x\right) \right\vert ^{q}+\left\vert f^{\prime }\left( a\right) \right\vert
^{q}}{s+1}\right) ^{\frac{1}{q}}\right.  \\
&&+\left. \left( \lambda _{3}(b,x,0,p,p)\right) ^{\frac{1}{p}}\left(
b-x\right) ^{2}\left( \frac{\left\vert f^{\prime }\left( x\right)
\right\vert ^{q}+\left\vert f^{\prime }\left( b\right) \right\vert ^{q}}{s+1}%
\right) ^{\frac{1}{q}}\right\} .
\end{eqnarray*}%
where $\lambda _{1}$, $\lambda _{2}$, $\lambda _{3}$ and $\lambda _{4}$ are
defined as in Theorem \ref{2.3}.
\end{theorem}

\begin{proof}
From Lemma \ref{2.1}, H\"{o}lder's inequality and the harmonically convexity
of $\left\vert f^{\prime }\right\vert ^{q}$ on $[a,b],$we have%
\begin{eqnarray*}
&&\left\vert f(x)-\frac{ab}{b-a}\dint\limits_{a}^{b}\frac{f(u)}{u^{2}}%
du\right\vert  \\
&\leq &\frac{ab\left( x-a\right) ^{2}}{b-a}\left( \dint\limits_{0}^{1}\frac{%
t^{p}}{\left( ta+(1-t)x\right) ^{2p}}dt\right) ^{\frac{1}{p}} \\
&&\times \left( \dint\limits_{0}^{1}\left[ t^{s}\left\vert f^{\prime }\left(
x\right) \right\vert ^{q}+(1-t)^{s}\left\vert f^{\prime }\left( a\right)
\right\vert ^{q}\right] dt\right) ^{\frac{1}{q}}
\end{eqnarray*}%
\begin{eqnarray*}
&&+\frac{ab\left( b-x\right) ^{2}}{b-a}\left( \dint\limits_{0}^{1}\frac{t^{p}%
}{\left( tb+(1-t)x\right) ^{2p}}dt\right) ^{\frac{1}{p}} \\
&&\times \left( \dint\limits_{0}^{1}\left[ t^{s}\left\vert f^{\prime }\left(
x\right) \right\vert ^{q}+(1-t)^{s}\left\vert f^{\prime }\left( b\right)
\right\vert ^{q}\right] dt\right) ^{\frac{1}{q}}
\end{eqnarray*}%
\begin{eqnarray*}
&\leq &\frac{ab}{b-a}\left\{ \left( \lambda _{1}(a,x,0,p,p)\right) ^{\frac{1%
}{p}}\left( x-a\right) ^{2}\left( \frac{\left\vert f^{\prime }\left(
x\right) \right\vert ^{q}+\left\vert f^{\prime }\left( a\right) \right\vert
^{q}}{s+1}\right) ^{\frac{1}{q}}\right.  \\
&&+\left. \left( \lambda _{3}(b,x,0,p,p)\right) ^{\frac{1}{p}}\left(
b-x\right) ^{2}\left( \frac{\left\vert f^{\prime }\left( x\right)
\right\vert ^{q}+\left\vert f^{\prime }\left( b\right) \right\vert ^{q}}{s+1}%
\right) ^{\frac{1}{q}}\right\} .
\end{eqnarray*}%
This completes the proof.
\end{proof}

\begin{corollary}
In Theorem \ref{2.7}, additionally, if $|f^{\prime }(x)|\leq M$, $x\in \left[
a,b\right] ,$ then inequality 
\begin{eqnarray*}
&&\left\vert f(x)-\frac{ab}{b-a}\dint\limits_{a}^{b}\frac{f(u)}{u^{2}}%
du\right\vert  \\
&\leq &\frac{ab}{b-a}M\left( \frac{2}{s+1}\right) ^{\frac{1}{q}}\left\{
\left( \lambda _{1}(a,x,0,p,p)\right) ^{\frac{1}{p}}\left( x-a\right)
^{2}\right.  \\
&&+\left. \left( \lambda _{3}(b,x,0,p,p)\right) ^{\frac{1}{p}}\left(
b-x\right) ^{2}\right\} 
\end{eqnarray*}%
holds.
\end{corollary}

\end{document}